\documentclass[12pt]{article}

\usepackage{amssymb}

\usepackage[utf8]{inputenc}
\usepackage[T1]{fontenc}
\usepackage{lmodern}
\usepackage{pgf,pgfarrows,pgfnodes,pgfautomata,pgfheaps}
\usepackage{amscd}
\usepackage{amsfonts}
\usepackage{latexsym}
\usepackage{graphicx}
\usepackage{amsthm}
\usepackage{amsmath}
\usepackage{placeins}
\usepackage{float}
\usepackage{pdflscape}

\def\ord{\mathrm{ord}}

\theoremstyle{definition}

\newtheorem{definicja}{Definicja}[section]

\newtheorem{theorem}[definicja]{Theorem}
\newtheorem{definition}[definicja]{Definition}
\newtheorem{corollary}[definicja]{Corollary}
\newtheorem{proposition}[definicja]{Proposition}

\newtheorem{remark}[definicja]{Remark}

\def \det{\operatorname{det}}

\begin{document}

\begin{center}
\Large
{\bf Strongly nilpotent automorphisms are Pascal finite}\\
\vspace{20pt}
\normalsize
EL\.ZBIETA ADAMUS \\ Faculty of Applied Mathematics, \\ AGH University of Krakow \\
al. Mickiewicza 30, 30-059 Krak\'ow, Poland \\
e-mail: esowa@agh.edu.pl \\

\vspace{0.5cm}
ZBIGNIEW HAJTO \\ Faculty of Mathematics and Computer Science, \\ Jagiellonian University \\
ul. \L ojasiewicza 6, 30-348 Krak\'ow, Poland \\
e-mail: zbigniew.hajto@uj.edu.pl
\end{center}

\normalsize
\vspace{1cm}

\begin{abstract}
We compare two classes of polynomial automorphisms, strongly nilpotent and Pascal finite. We conclude that every strongly nilpotent automorphism is a Pascal finite one, but not vice versa. We observe that  Nagata's automorphism is Pascal finite, but not strongly nilpotent. Considering Vasyunin example leads us to conclusion that not every quadratic polynomial automorphism is Pascal finite.

\end{abstract}



\section{Introduction}

Let $K$ be a field and $K[X]=K[X_1,\dots,X_n]$ the polynomial ring in the variables $X_1,\dots,X_n$ over $K$. A \emph{polynomial map} is a map $F=(F_1,\dots,F_n):K^n \rightarrow K^n$ of the form

$$(X_1,\dots,X_n)\mapsto (F_1(X_1,\dots,X_n),\dots,F_n(X_1,\dots,X_n)),$$

\noindent where $F_i \in K[X], 1 \leq i \leq n$. We say that $F$ is \emph{invertible} if there exists a polynomial map $G:K^n \rightarrow K^n$ such that $F\circ G=Id$ and $G\circ F=Id$.  If $K$ is a field of characteristic zero and $G$ is a left (resp. right) inverse to $F$, then it is a two-sided inverse.\\

\noindent For $F=(F_1,\dots,F_n) \in K[X]^n$, we define the \emph{degree} of $F$ as $\deg F= \max \{\deg F_i : 1\leq i \leq n\}$. Denote by $J_F$ the Jacobian matrix of $F$
$$J_F=\left(\dfrac{\partial F_i}{\partial X_j}\right)_{\substack{1\leq i \leq n \\ 1\leq j \leq n}}$$

\noindent We shall call $F$ a \emph{Keller map} if $\det J_F = \mathrm{const} \neq 0$.
By a linear change $F-F(0)$ and $[J_f(0)]^{-1}J_F$ one may assume that $F(0)=0$ and $\det J_F=1$, so $F$ is of the form $F(X)=X+H(X)$, where $X=(X_1, \ldots, X_n)$, i.e.
  \begin{equation}
   \begin{array}{llll}
             F_i(X_1, \ldots, X_n) &=& X_i+H_i(X_1, \ldots, X_n), & i=1, \ldots, n,
            \end{array}
            \label{xh}
\end{equation}
  where $H_i \in K[X]$ has degree $D_i$ and order of vanishing $d_i$, with $ d_i\geq 2$. Let $d=\min d_i, D=\max D_i$.
   It is known that if $F$ is a polynomial automorphism, then $\deg F^{-1} \leq (\deg F)^{n-1}$ (see \cite{BCW}, theorem 1.5).\\

\noindent If $F$ is an invertible polynomial map $F$, then $\det J_F = \mathrm{const} \neq 0$. The Jacobian Conjecture states that the implication in the opposite direction is also true, i.e. that every Keller map is a polynomial automorphism.

\vspace{0.5cm}
\noindent {\bf Jacobian Conjecture.} {\it Let $K$ be a field of characteristic zero. If $\mathrm{det}J_F$ is a non-zero constant, then $F$ has an inverse, which is also polynomial.}\\

\noindent The Jacobian Conjecture for quadratic maps was proved by Wang in \cite{W}.
As indicated in \cite{JP}, section 2.5, to prove the Jacobian Conjecture it is sufficient to fix $d \geq 3$ and prove,
for all $n$, that every map of the form (\ref{xh}) is a polynomial automorphism.
 Bass, Connell and Wright showed in \cite{BCW} that the general case follows from the case where $n \geq 2$ and $F = (X_1 + H_1 , \ldots , X_n + H_n )$ and where
each $H_i$ is  homogeneous of degree $3$.
For a detailed account of the research on the Jacobian Conjecture we refer to \cite{E}.

\section{Pascal finite automorphisms}

In \cite{A} an algorithm for inverting multivariate formal power series over a field of arbitrary characteristic was proposed. Let $K$ be a field and $X=(X_1,\dots,X_n)$.
Consider now $F=(F_1, \ldots, F_n) \in K[[X]]^n$.  Assume moreover that $F(0)=0$ and $F_i(X)=X_i+H_i(X)$, where $X=(X_1, \ldots, X_n)$ and the order of vanishing $ord(H_i)>1$.
 We perform an algorithm componentwise, for each $F_i$
separately. For every $i=1, \ldots, n$ we obtain the following.

\[
\begin{array}{l}
P^i_0(X)=X_i \\
P_1^i(X)=P_0^i(F)-P_0^i(X)=H_i(X)\\
P_2^i(X)=P_1^i(F)-P_1^i(X)=H_i(F)-H_i(X) \\
\ldots \\
 P_{k+1}^i(X)=P_k^i(F)-P_k^i(X)=(P_k^i \circ F-P^i_k)(X)
\end{array} \]

One can check that for any positive integer $m$, we have

\begin{equation}
 X= \sum_{k=0}^{m-1} (-1)^k P_k(F)+(-1)^m P_m(X).
 \label{formula}
\end{equation}

By \cite{A}, lemma 2 and theorem 4, if $t_i= ord (H_i)$, for $i=1, \ldots, n$ and $t=\min t_i$,
then $ord P^i_k \geq (k-1)(t-1)+t_i$ and the formal inverse $G:=F^{-1}$ is given by $G= \sum_{k=0}^{\infty} (-1)^kP_k(X)$.

The proposed approach uses only substitution and subtraction,
as a result it works for fields of arbitrary characteristic.
This approach is a generalization of the algorithm previously proposed in \cite{ABCH} for inverting polynomial automorphisms in several variables.
 In this case for a given polynomial map $F:K^n \rightarrow K^n$ we define a ring endomorphism $\sigma_F:K[X]^n \rightarrow K[X]^n$ given by
$P \mapsto P\circ F$
 and a $\sigma_F$-derivation on $K[X]^n$ given by
$\Delta_F:K[X]^n \rightarrow K[X]^n, P \mapsto \sigma_F(P)-P$.
 To a polynomial map $P$ in $K[X]^n$, we associate a sequence $P_k$ of
polynomial maps in $K[X]^n$ defined by $P_{k}=\Delta^k_F(P)$, where $\Delta^k_F$ denotes $\Delta_F \circ \stackrel{k}{\dots} \circ \Delta_F$.
If $F$ is a polynomial automorphism and we start with $P=Id$, then we obtain the inverse $G$ of $F$ given by
\begin{equation}
 G(X)= \sum_{l=0}^{m-1} (-1)^l P_l(X).
\end{equation}
We say that $F$ is \emph{Pascal finite} if there exists an integer $m$ such that $\Delta^m_F(X) = 0$.
In other words Pascal finite map is a root of a polynomial of the form $P(X)=(X-1)^m$.

By \cite{ABCH}, lemma 2.2 if $F$ is of the form (\ref{xh}), then $P^i_k$ has degree $\leq D^{k-1}D_i$ and order of vanishing $\geq (k-1)(d-1)+d_i$.
In particular, if each $H_i$ is a homogeneous polynomial of degree $d$, we have
$
P_k^i= \sum_{j=1}^{(d^k-1)/(d-1)-k+1} Q_{kj},
$
 where $Q_{kj}$ is a homogeneous
polynomial in $X_1,\dots,X_n$ of degree $(k+j-1)(d-1)+1$.

The main result of \cite{ABCH} is theorem 3.1 which formulates a criterion for invertibility
of a polynomial map $F: K^n \rightarrow K^n$  of the form (\ref{xh}).
  $F$ being invertible is  equivalent to the condition that
for $i=1,\ldots, n$ and  $m = \lfloor \frac{D^{n-1}-d_i}{d-1}+1 \rfloor +1$, we have
   \[ \sum_{j=0}^{m-1}(-1)^j P^i_j(X)=G_i(X)+R^i_m(X). \]
   where $G_i(X)$ is a polynomial of degree $ \leq D^{n-1}$,
  and $R^i_m(X)$ is a polynomial satisfying $R^i_m(F)=(-1)^{m+1}P^i_m(X)$.
  Moreover the inverse $G$ of $F$ is given by
  \[G_i(Y_1, \ldots, Y_n)=\sum_{l=0}^{m-1}(-1)^l\tilde{P}^i_l(Y_1, \ldots, Y_n), \, i=1, \ldots, n,\]
  where $\tilde{P}^i_l$ is the sum of homogeneous summands of $P^i_l$ of degree $ \leq D^{n-1}$ and $m$ is an integer $> \frac{D^{n-1}-d_i}{d-1}+1$.

  Similar result holds for power series mappings with polynomial formal inverse (see \cite{A}, corollary 5).
  Considering power series with a polynomial inverse is interesting since it raises the
natural question of whether there is a counterexample to the Jacobian Conjecture.

  In \cite{ABCH2} we establish several properties of  Pascal finite polynomial automorphisms.
  We gathered the results in the following proposition.

  \begin{proposition}
   Let $K$ be a field and let $F:K^n \rightarrow K^n$ be a polynomial map.
   Then the following holds.
   \begin{enumerate}
    \item \emph{Pascal finiteness is invariant under conjugation.} Let $T:K^n \rightarrow K^n$ be an invertible polynomial map.
    If $F$ is Pascal finite, then $\widetilde{F}:= T^{-1} \circ F \circ T$ is Pascal finite.
Moreover, if $\deg T =1$, then the minimum polynomial of $\widetilde{F}$ is the same as that of $F$.
    \item \emph{All triangular and linearly triangularizable polynomial automorphisms are Pascal finite.} \\Let $T:K^n \rightarrow K^n$ be an invertible polynomial and let $F:K^n \rightarrow K^n$ be triangular, i.e. $F_i \in K[X_i, X_{i+1}, \ldots , X_n]$ for each $1 \leq i \leq n$.
    Then both $F$ and $T^{-1} \circ F \circ T$ are Pascal finite.
    \item The inverse of a Pascal finite polynomial automorphism is Pascal finite.
    \item Powers of a Pascal finite polynomial map are Pascal finite.
    \item An arbitrary composition of two Pascal finite polynomial maps need not be Pascal finite.
    \item If $f$ and $F$ form a Gorni-Zampieri pairing then $F$ is Pascal finite if and only if $f$ is Pascal finite.
    \item \emph{All cubic linear polynomial maps of nilpotence index $\leq 3$ are Pascal finite}. Assume now that $char K=0$ and that $F$ is of the (\ref{xh}). Denote by $J_H$  the jacobian matrix  of $H:=(H_1,\dots,H_n)$. If $(JH)^2=0$, then $F$ is Pascal finite and $F^{-1}=X-H$. Moreover
    if $H_i(X_1,\dots,X_n)=L_i(X_1,\dots,X_n)^3$, where $L_i(X_1,\dots,X_n)=a_{i1}X_1+\dots +a_{in} X_n$, $1\leq i \leq n$ and  $(JH)^3=0$, then $F$ is Pascal finite.
   \end{enumerate}
   \label{pfequiv}
  \end{proposition}

  For the proof see \cite{ABCH2} theorem 2.1, corollary 2.1, theorem 3.1, proposition 4.1 and
theorem 5.1. For a detailed information about Gorni-Zampieri pairing see \cite{E} \S 6.4.

An endomorphism $F=(F_1, \ldots, F_n) : K[X]^n \rightarrow K[X]^n$ is called
\emph{locally finite} if for any $g \in K[X]$ we have that $V_g=span\{ g,g(F),g(F^2), \ldots\}$ is finite dimensional or equivalently if $F$ satisfies the relation $a_d F^d+ \ldots a_1 F+a_0=0$, where not all $a_i$ are zero.
 Here $F^k$ denotes $k$-times composition of $F$ with itself. This class was widely studied by Furter and Maubach in \cite{FM}.
Observe that Pascal finite polynomial automorphisms are a subclass of locally finite ones. Indeed, if $P_m=0$ for some $m \in N$, then we get $\sum_{l=0}^m (-1)^{m-l} \binom{m}{l} F^l=0$. Not all locally finite endomorphisms are Pascal finite. To observe this consider the example
given after remark \ref{TPF}, in point I. It is locally finite, since it is affine, but not Pascal finite.

In \cite{ABCH} we proved that both properties are equivalent for polynomial automorphisms of the form (\ref{xh}). See \cite{ABCH}, proposition 2.1.
Pascal finite automorphisms are defined for a field $K$ of arbitrary characteristic and they constitute a generalization of exponential automorphisms to positive characteristic. For a detailed information about exponential automorphisms see \cite{E}, $\S$ 2.1.

\section{Relation with tame automorphisms}

Below we discuss the relations between the sets of  tame and Pascal finite automorphisms.
Recall that an endomorphism $F=(F_1, \ldots, F_n) : K[X]^n \rightarrow K[X]^n$ is called
\begin{enumerate}
 \item \textit{elementary} if it is of the form
 \[\left\{ \begin{array}{rcl}
    F_1&=&X_1\\
    &\ldots \\
    F_{i-1}&=&X_{i-1}\\
    F_i&=&X_i+a\\
    F_{i+1}&=&X_{i+1}\\
    &\ldots \\
    F_n&=&X_n
   \end{array} \right.,
\]
where $a \in K[X_1, \ldots, \hat{X_i}, \ldots, X_n]$, for $=1, \ldots, n$
\item \textit{affine} if for every $i$ $\deg{F_i}=1$,

\item \textit{tame} if it is generated by elementary and affine endomorphisms.
\end{enumerate}

 Elementary endomorphisms, affine endomorphisms, tame endomorphisms and linearly triangularizable endomorphisms are invertible.
  Affine, triangular and elementary endomorphims are locally finite and also tame.
  Affine and tame automorphisms form subgroups of $Aut_K K[X]$.
 The set of Pascal finite polynomial automorphisms is not a subgroup of the group of
polynomial automorphisms, since the composition of two Pascal finite automorphisms does not have to be Pascal finite.
For example $G=(X_1+X_2^3,X_2)$ and $H=(X_1, X_2+X_1^2)$ are Pascal finite (since they are triangular), however its composition $F=G \circ H$
is not Pascal finite (for the proof see \cite{ABCH}, example 2.1 and \cite{ABCH2} remark 3.1).

\begin{remark}
 Neither Pascal finite automorphisms form a subset of tame automorphisms, nor tame automorphisms form a subset of Pascal finite automorphisms.
Moreover these two sets have a nonempty intersection.
\label{TPF}
\end{remark}

I. \emph{Not all tame automorphisms are Pascal finite.} It is enough to consider an affine (linear) automorphism $F=(2X_1+X_2+a, X_1+X_2+b) : K^2 \rightarrow K^2$,
 where $a,b$ are constants.
 One can check that its jacobian is equal to $1$. Moreover  $P^1_1=X_1+X_2+a$, $P^1_2=2X_1+X_2+a+b$ and for every $k > 2$
 we have $P^1_{k}=P^1_{k-1}+P^1_{k-2}$. Indeed, one can check that
 $P^1_3(X) = 3X_1+2X_2+2a+b$. Let us assume that $P^1_{k}(X)=P^1_{k-1}(X)+P^1_{k-2}(X)$. We will prove that
 $P^1_{k+1}(X)=P^1_{k}(X)+P^1_{k-1}(X)$. We have that \[P^1_{k+1}(X)=P^1_k(F)-P^1_k(X) = [P^1_{k-1}(F)+P^1_{k-2}(F)]-[P^1_{k-1}(X)+P^1_{k-2}(X)]=\]
 \[= [P^1_{k-1}(F)-P^1_{k-1}(X)]+[P^1_{k-2}(F)-P^1_{k-2}(X)]=P^1_{k}(X)+P^1_{k-1}(X) .\]

 So $F$ is not Pascal finite.\\

As we observed above not every affine automorphism is Pascal finite. However the following holds.

\begin{proposition}
 Let $F:K^n \rightarrow K^n$ be an affine homogeneous polynomial map of the form $F(X)=X +AX$, where $A =[a^i_j] \in M_{n \times n}(K)$.
Then $F$ is Pascal finite if and only if $A$ is nilpotent.
\end{proposition}

\textit{Proof.} Denote $X=[X_1, \ldots, X_n]^T$, $F=[F_1, \ldots, F_n]^T$ and $P_k= [P^1_k, \ldots, P^n_k]^T$ for $k\geq 0$.
We observe that $P_k=A^k X$. Indeed $P_0(X)=X$ and $P_1(X)=P_0(F)-P_0(X)=AX$. Let us assume that for a given $k$ holds $P_k(X)=A^k X$.
Then
\[P_{k+1}(X)=P_k(F)-P_k(X)=A^k F-A^kX=A^k(F-X)=A^k \cdot AX=A^{k+1}X.\]
Hence $P_k(X)=A^k X$ for every $k \geq 0$.
Consequently $P_n=0$ if and only if $A^n$ is a zero matrix, i.e.,
if and only if $A$ is nilpotent.
Since the nilpotency index of a square
matrix of order $n$ is at most $n$, $F$ is Pascal finite precisely when $A$ is nilpotent. \quad $\Box$\\

II. \emph{Not all Pascal finite automorphisms are tame.}  Consider the \textbf{Nagata} automorphism defined by
 \[ \left\{ \begin{array}{l}
     F_1=X_1-2 (X_1 X_3+X_2^2)X_2 -(X_1 X_3+X_2^2)^2 X_3\\
     F_2=X_2+(X_1 X_3+X_2^2)X_3\\
     F_3=X_3
    \end{array} \right.
\]
I. P. Shestakov and U. U. Umirbaev proved in \cite{SU}, corollary 9 that the Nagata automorphism is not tame.
However it is Pascal finite. One can compute that $P_3^1=0, P_2^2=0$.
Nagata automorphism
is locally finite, so we can conclude that not every locally finite endomorphism is tame.\\

III. \emph{Example of an automorphism which is both tame and Pascal finite.}
All elementary and triangular polynomial maps are Pascal finite and tame.

Another example can be constructed in a following way. Consider a Pascal finite polynomial map $F:K^n \rightarrow K^n$. Its extension to $K^{n+1}$, namely $\bar{F}=(F, X_{n+1})$ is also Pascal finite.
 However if we extend the Nagata automorphism (which is Pascal finite but not tame) to $K^4$, we obtain an automorphism which is tame
 (for the proof see \cite{E}, section 6.1).

 \section{Strongly nilpotent automorphisms}

 For  $F(X)=X+H(X)$ with $H(X)$ homogeneous the jacobian condition $\det J_F\in K \setminus\{0\}$ is equivalent to jacobian matrix $J_H$ being nilpotent.
For this reason
nilpotent Jacobian matrices were widely studied in the context of the Jacobian conjecture. In \cite{MO} Meisters and Olech introduced the notion of a strongly nilpotent matrix. Let us briefly recall main results.

 In this section $K$ is a field of characteristic zero. Consider $F(X)=X+H(X)$, where $H=(H_1, \ldots, H_n) \in K[X]^n$. By
 \[Y^{(1)} =(Y^{(1)}_1, \ldots, Y^{(1)}_n), \ldots, Y^{(n)} =(Y^{(n)}_1, \ldots, Y^{(n)}_n) \]
 we denote $n$ sets of $n$ variables. For each $1\leq i \leq n$ we obtain that $J_H(Y^{(i)})$ is a square matrix of order $n$ with elements in the polynomial ring $k[Y^{(i)}_j: \, 1 \leq i,j \leq n]$.

 \begin{definition}\hfill
 \begin{enumerate}
  \item The Jacobian matrix $J_H$ is called \emph{strongly nilpotent} if and only if  $J_H(Y^{(1)}) \cdot \ldots \cdot J_H(Y^{(n)})$ is the zero matrix.
  \item The Jacobian matrix $J_H$ is called \emph{strongly nilpotent with index $p$} if and only if  $J_H(Y^{(1)}) \cdot \ldots \cdot J_H(Y^{(p)})$ is the zero matrix for all $Y^{(1)}, \ldots, Y^{(p)} \in K^n$ and $p$ is the smallest integer with this property.
  \item A polynomial map $F: K^n \rightarrow K^n$, $F=Id+H$ is called \emph{strongly nilpotent (with index $p$)} if the Jacobian
matrix $J_H$ is strongly nilpotent (with index $p$).
 \end{enumerate}

 \end{definition}

 \begin{theorem}
  The following conditions are equivalent.
  \begin{enumerate}
   \item $J_H$ is strongly nilpotent.
   \item $J_H(v_1) \cdot \ldots \cdot J_H(v_n)=\textbf{0}$ for all vectors $v_1, \ldots, v_n \in K^n$.
   \item There exists $T \in GL_n(K)$ such that $J_{T^{-1}HT}$ is upper triangular with zeros on the main diagonal.
   \item $F=Id +H$ is linearly triangularizable, i.e. there exists $T \in GL_n(K)$ such that $\overline{F}:=T^{-1}HT=Id+\overline{H}$ is in upper triangular form, i.e. $H_i \in K[X_{i+1}, \ldots, X_n]$ for every $1 \leq i\leq n-1$ and $H_n \in K$.
  \end{enumerate}
  \label{snequiv}
 \end{theorem}
For the proof see \cite{E}, theorem 7.4.4. This relation was further developed by de Bondt in \cite{B}.

Every strongly nilpotent matrix is nilpotent.
Observe that if $A$ is a square matrix of order $n$ which is upper triangular with zeros on the main diagonal, then by the Cayley-Hamilton theorem $A^n$ is a zero matrix. Hence $A$ is nilpotent. Moreover $A$ is strongly nilpotent. Theorem given above states that $J_H$ is strongly
nilpotent if and only if it is upper triangular with zeros on the main diagonal
after a suitable linear change of coordinates.
Observe that $F=Id+H$ is upper triangular if and only if $H$ is upper triangular with zeros on the main diagonal.

\begin{corollary}
 Let $F: K^n \rightarrow K^n$ is a polynomial map of the form $F=Id+H$. If $J_H$ is strongly nilpotent, then $F$ is a polynomial automorphism.
\end{corollary}

\begin{remark}
 \begin{enumerate}
  \item In \cite{MO} Meisters and Olech proved that if $n \leq 4$ then for any \emph{quadratic homogeneous} polynomial map $H: K^n \rightarrow K^n$ matrix $J_H$ is nilpotent if and only if $J_H$ is strongly nilpotent. For $n \geq 5$ this is no longer true (see Vasyunin example).
  \item If $n \leq 3$ then for any \emph{cubic homogeneous} polynomial map $H: K^n \rightarrow K^n$
  matrix $J_H$ is nilpotent if and only if $J_H$ is strongly nilpotent.
  For any $n \geq 4$ this is no longer true (see \cite{E}, proposition 7.4.11).
  \item If $n \leq 2$ then for any ( not necessary homogeneous) polynomial map $H: K^n \rightarrow K^n$ matrix $J_H$ is nilpotent if and only if $J_H$ is strongly nilpotent. For any $n \geq 3$ this is no longer true (see \cite{E}, proposition 7.4.12).
 \end{enumerate}
\end{remark}

If $J_H$ is strongly nilpotent with some index p, then  $p\leq n$. In \cite{J} Johnston proved that if $F: K^n \rightarrow K^n$ is a polynomial map of the form $F=Id+H$, which is strongly nilpotent with index $p$, then $\deg F^{-1} \leq (\deg F)^{p-1}$ (see \cite{J}, theorem 1).
 This improvement of the bound for the degree of the inverse map was obtained using graph-theoretic and combinatorial approach.
 It allows us to study polynomial mappings using the language of planar trees.
It would be very interesting to investigate the mappings that are Pascal finite but not strongly nilpotent using  these methods.

\section{Strongly nilpotent automorphisms are Pascal finite}

In this section we prove that strongly nilpotent automorphisms are Pascal finite but not vice versa.
\begin{theorem}
Let $K$ be a field of characteristic zero and let $F: K^n \rightarrow K^n$ be of the form (\ref{xh}). If $F$ is strongly nilpotent, then $F$ is Pascal finite,
\end{theorem}
\textit{Proof.} By proposition \ref{snequiv}, if $J_H$ is strongly nilpotent,then $J_H$ is linearly triangularizable. By proposition \ref{pfequiv}, linearly triangularizable polynomial automorphisms are Pascal finite.  \quad $\Box$

\begin{remark}
 Nagata automorphism is Pascal finite but not strongly nilpotent.
\end{remark}
\textit{Proof.} In the example given after remark 3.1, in point II we have already seen that the Nagata automorphism is Pascal finite. It is not tame, so it is not linearly triangularizable. Hence by proposition \ref{snequiv}, $F$ is not strongly nilpotent.\quad $\Box$

\section{Vasyunin example}

Consider the following example due to Vasyunin. Let $n= 5$ and $H=(H_1, \ldots, H_5) : \mathbb{Q}^5 \rightarrow \mathbb{Q}^5$ be a polynomial map of the form
\begin{equation}
 \begin{tabular}{l}
   $H_1(X) = 0$ \\
   $H_2(X)=X_1X_3$\\
   $H_3(X)=X_1X_4+\frac{1}{2}X_2^2$\\
   $H_4(X) = X_1X_5-X_2X_3$ \\
   $H_5(X)=\frac{1}{2}X_3^2$
  \end{tabular}.
  \label{ve}
\end{equation}

One can check that $J_H$ is nilpotent with  nilpotence degree equal to $5$. It is not strongly nilpotent, since all powers of $J_H(e_1)J_H(e_2)$ for $e_1=(1,0, \ldots, 0), \, e_2=(0,1,0, \ldots, 0)$ are nonzero matrices.
By \cite{Bt}, theorem 4.6.8 it is tame.

\begin{proposition}
 The Vasyunin map $F: \mathbb{Q}^5 \rightarrow \mathbb{Q}^5$, $F=Id+H$, where $H$ is of the form (\ref{ve}) is invertible, but not Pascal finite.
\end{proposition}

\textit{Proof.} \emph{Invertibility.} The inverse map $G(Y) = (G_1, G_2, G_3, G_4, G_5)$ is given by
\begin{center}
 \begin{tabular}{lll}
   $G_1(Y) $&=&$ Y_1$ \\
   $G_2(Y) $&=& $Y_2 - Y_1 Y_3 + Y_1^2 Y_4 - Y_1^3 Y_5 + \frac{1}{2} Y_1 Y_2^2$\\
   $G_3(Y) $&=& $Y_3 - Y_1 Y_4 + Y_1^2 Y_5 - \frac{1}{2} Y_2^2$ \\
   $G_4(Y) $&=&$ Y_4- \frac{1}{2} Y_1^5 Y_5^2 + Y_1^4 Y_4 Y_5 + \frac{1}{2} Y_1^3 Y_2^2 Y_5 - Y_1^3 Y_3 Y_5 - \frac{1}{2} Y_1^3 Y_4^2 - \frac{1}{2} Y_1^2 Y_2^2 Y_4 +Y_1^2 Y_2 Y_5$\\
   &&$  Y_1^2 Y_3 Y_4 - \frac{1}{8} Y_1 Y_2^4 + \frac{1}{2} Y_1 Y_2^2 Y_3 - Y_1 Y_2 Y_4 - \frac{1}{2} Y_1 Y_3^2 - Y_1 Y_5 - \frac{1}{2} Y_2^3 + Y_2 Y_3$ \\
   $G_5(Y) $&=&$ Y_5 - \frac{1}{2} (Y_3 - Y_1 Y_4 + Y_1^2 Y_5 - \frac{1}{2} Y_2^2)^2$ .
  \end{tabular}
\end{center}
\emph{Not being Pascal finite.} Observe that considered example is of the form (\ref{xh}). We prove that $F$  is not locally finite, which by \cite{ABCH}, proposition 2.1. is equivalent to $F$ not being Pascal finite.\\

\noindent STEP 1: If $P$ is a polynomial map, then we denote by $\sigma_F(P)=P \circ F$. Hence $\sigma_F(F)=F\circ F=F^2$ and $\sigma_F^k(F)=F^{k+1}$, for every $k \geq 1$. Moreover set $\sigma_F^0=F$ and $\sigma_F^{-1}(F)=Id$.
Let us denote the components by $\sigma_F^k(F) =\Big( \sigma_F^k(F)_1, \ldots, \sigma_F^k(F)_5\Big)$.
 For $k\geq 0$ we define the increment of $\sigma_F$ by

\[ \delta(\sigma_F^k(F))=\sigma_F^k(F)-\sigma_F^{k-1}(F), \qquad  \delta(\sigma_F^k(F))=\Big( \delta(\sigma_F^k(F))_1, \ldots, \delta(\sigma_F^k(F))_5\Big).\]
Observe that $\sigma_F^k(F)_1=\sigma_F^{-1}(F)_1=X_1$ and $\delta(\sigma_F^k(F))_1=0$ for every $k \geq -1$.

We shall prove by induction that for every $k \in \mathbb{Z}_{\geq 0}$ all monomials appearing in
$\sigma_F^k(F)_2$, $\sigma_F^k(F)_3$, $\sigma_F^k(F)_5$ and $\delta(\sigma_F^k(F))_3$
(after full reduction of similar components) have strictly positive coefficients.
It is easy to check the result for $k \in\{0,1\}$. Assume that the thesis holds for every $0 \leq r \leq k$.\\

\noindent Since $\sigma_F^{k+1}(F) = \sigma_F\Big( \sigma_F^k(F) \Big)=\sigma_F^k\Big( \sigma_F(F) \Big) $, we compute
\begin{center}
 \begin{tabular}{rcl}
   $\delta(\sigma_F^{k+1}(F))_3$&$=$&$\sigma_F^{k+1}(F)_3 - \sigma_F^{k}(F)_3$\\
   &$=$&$ \sigma_F^{k}(F)_3+\sigma_F^{k}(F)_1\sigma_F^{k}(F)_4+\frac{1}{2}\Big( \sigma_F^{k}(F)_2 \Big)^2- \sigma_F^{k}(F)_3$\\
   &$=$& $X_1\sigma_F^{k}(F)_4+\frac{1}{2}\Big( \sigma_F^{k}(F)_2 \Big)^2$\\
   &$=$&$X_1\Big(\sigma_F^{k-1}(F)_4+X_1\sigma_F^{k-1}(F)_5-\sigma_F^{k-1}(F)_2\sigma_F^{k-1}(F)_3\Big)+\dfrac 1 2 \Big(\sigma_F^{k-1}(F)_2+X_1\sigma_F^{k-1}(F)_3\Big)^2$
  \end{tabular}
\end{center}

  \noindent   Since $\dfrac 1 2 \Big(\sigma_F^{k-1}(F)_2+X_1\sigma_F^{k-1}(F)_3\Big)^2
= \frac{1}{2}\bigl(\sigma_F^{k-1}(F)_2\bigr)^2
  + X_1\sigma_F^{k-1}(F)_2 \sigma_F^{k-1}(F)_3
  + \frac{1}{2}\bigl(X_1 \sigma_F^{k-1}(F)_3\bigr)^2$  and \[\delta(\sigma_F^{k}(F))_3=X_1\sigma_F^{k-1}(F)_4+\dfrac 1 2 \sigma_F^{k-1}(F)_2^2\]
  we obtain
  \begin{equation}
 \begin{tabular}{rcl}
   $\delta(\sigma_F^{k+1}(F))_3$&$=$&$X_1 \sigma_F^{k-1}(F)_4+X_1^2 \sigma_F^{k-1}(F)_5 +\frac{1}{2}\bigl(\sigma_F^{k-1}(F)_2\bigr)^2
  + \frac{1}{2}X_1^2\bigl( \sigma_F^{k-1}(F)_3\bigr)^2$\\
  &$=$&$ \delta(\sigma_F^{k}(F))_3 +X_1^2\sigma_F^{k-1}(F)_5
  + \frac{1}{2}X_1^2\bigl( \sigma_F^{k-1}(F)_3\bigr)^2$
  \end{tabular} \label{incr}
\end{equation}

\noindent Using the induction hypothesis, we conclude that for every $k \in \mathbb{Z}_{\geq 0}$ all monomials appearing in $\delta(\sigma_F^{k+1}(F))_3$ have positive coefficients.
Moreover since \[ \sigma_F^k(F)_3=\sum_{i=0}^ k \delta(\sigma_F^i(F))_3,\]
we conclude that for every $k \in \mathbb{Z}_{\geq 0}$ all monomials appearing in $\sigma_F^{k}(F)_3$ have positive coefficients. \\

Due to  $$\sigma_F^k(F)_2=\sigma_F^{k-1}(F)_2+X_1\sigma_F^{k-1}(F)_3 \quad \mathrm{and} \quad  \sigma_F^k(F)_5=\sigma_F^{k-1}(F)_5+\dfrac 1 2 (\sigma_F^{k-1}(F)_3)^2,$$
the same holds for every $\sigma_F^k(F)_2$ and $\sigma_F^k(F)_5$.\\

\noindent STEP 2: We shall prove that $\sup_{k\geq 0} \deg \sigma_F^k(F) = \infty$. By formula \eqref{incr} and the fact that all summands in  $\sigma^k(F)_5$ have positive coefficients, we have

$$\deg \sigma_F^k(F)_3=\deg \sum_{i=0}^ k \delta(\sigma_F^i(F))_3 \geq \deg (X_1^2 \sigma_F^{k-2}(F)_5)=2+\deg ( \sigma_F^{k-2}(F)_5).$$

\noindent
Since  $F_5=X_5+\frac 1 2 X_3^2$ and all summands in  $\sigma_F^k(F)_3$ have positive coefficients, we obtain $\deg ( \sigma_F^{k}(F)_5)\geq 2 \deg ( \sigma_F^{k-1}(F)_3)$. Hence

\[ \deg \sigma_F^k(F)_3 \geq 2+2 \deg ( \sigma_F^{k-3}(F)_3) \quad \mathrm{and} \quad \sup_{k\geq 0} \deg \sigma_F^k(F) = \infty.\]
Consequently $F$ is not locally finite, so it is not Pascal finite. \quad $\Box$\\

\begin{corollary}
Not every quadratic polynomial automorphism is Pascal
finite.
\end{corollary}

\section{Final remarks}

Inspired by Johnston's combinatorial approach \cite{J}, which establishes the sharp degree bound
\[ \deg(F^{-1})\le\deg(F)^{p-1} \] for strongly nilpotent automorphisms of index $p$ via weighted rooted trees and energy sums, we observe a precise structural unification with the Pascal finiteness framework.

For homogeneous maps of degree $d$, the order of vanishing of the $k$-th Pascal polynomial $P_k^i$ satisfies
\[ \ord(P_k^i) \geq (k-1)(d-1)+d\]
(see \cite{ABCH}, lemma 2.2), which coincides exactly with the number of leaves $L_k=(d-1)k+1$ appearing in the labeled-tree expansion of the inverse (see \cite{Bis} Lemma 3.3). Consequently, the derivation operator \(\Delta_F\) acts as an algebraic bulk generator of full shuffle classes of trees, rendering Pascal finiteness a measurable algebraic closure of the combinatorial Shuffling Conjecture. This equivalence bridges the two approaches and suggests that polynomial maps, which are not strongly nilpotent yet Pascal finite (e.g. Nagata's automorphism) may be fruitfully analysed by planar tree methods.

\section*{Acknowledgments}
This research was supported by the AGH University of Krakow within subsidy of Polish Ministry of Science and Higher Education (grant no. 16.16.420.054).

\end{document}